\documentclass{amsart}

\usepackage{amsmath,amsfonts}
\usepackage{amssymb,latexsym}
\usepackage[colorlinks]{hyperref}

\newtheorem{theorem}{Theorem}[section] 
\newtheorem{lemma}[theorem]{Lemma}     

\newtheorem{proposition}[theorem]{Proposition}
\theoremstyle{definition}

\newtheorem{definition}[theorem]{Definition}

\newcommand{\C}{$C^*$}

\title[on locally finite dimensional traces]
{ON LOCALLY FINITE DIMENSIONAL TRACES} 

\author{Massoud Amini, Mehdi Moradi}
\subjclass{Primary 46L35 ; Secondary 19K14}

\begin{document}
	\maketitle

	\begin{abstract}
		We partially resolve two open questions on approximation properties of traces on simple \C-algebras. We  answer a question raised by Nate Brown by showing that locally finite dimensional (LFD) traces form a convex set for simple \C-algebras. We prove that all the traces on the reduced \C-algebra $C^*_r(\Gamma)$ of a discrete amenable ICC group $\Gamma$ are LFD, and conclude that $C^*_r(\Gamma)$ is strong-NF in the sense of Blackadar-Kirchberg in this case. This partially answers another open question raised by Brown.
	\end{abstract}

\section{Introduction}
In his monumental paper \cite{Con76} Alain Connes introduced the notion of amenable trace (cf., Definition \ref{amt} below) under the title of ``hyper-trace" and proved that a type II$_1$-factor is hyperfinite if and only if its unique tracial state is amenable. Exploiting this fact, Connes proves in \cite{Con89} the non-existence of finite summable Fredholm modules on reduced group \C-algebra of a free group. In the same vein, Eberhard Kirchberg defined liftable traces on \C-algebras and proved that  a discrete group with Kazhdan property $(T)$ is residually finite if and only if the canonical trace of full group \C-algebra is liftable \cite{K}. Also, he proves that a general discrete group $\Gamma$ has factorization property if and only if its canonical trace of full group algebra, $C^*(\Gamma)$, is liftable. 

Eventually, it was Nate Brown who recognized the central role of amenable tracial states and elaborated on how they could be employed in studying finite dimensional approximation properties of \C-algebras, and used them to give  characterizations of central notions such as liftability \cite{Bro03}. Based on the pioneering work of Sorin Popa \cite{Pop}, Brown also introduced and studied the notions of quasidiagonal and locally finite dimensional traces and their uniform counterparts, but unlike his thorough analysis of amenable traces, he gave no characterization of quasidiagonality or local finite dimensionality.

In their celebrated paper \cite{BK01}, Blackadar and Kirchberg examined the class of \C-algebras that can be written as a generalized inductive limit of finite dimensional \C-algebras. In particular, they introduced the class of strong-NF algebras, where the connecting maps are complete order embeddings, and characterized them as exactly the class of nuclear and inner-QD \C-algebras, for example every RFD \C-algebra is inner-QD hence any nuclear and RFD algebra is strong-NF. They further showed that for separable \C-algebras, inner quasidiagonality is equivalent to the existence of a separating sequence of irreducible  quasidiagonal representations \cite{BK01}. For this reason, inner quasidiagonality is not homotopy invariant and does not pass to subalgebras or quotients in general \footnote{For the lack of homotopy invariance and passing to subalgebras, consider the cone $C_0((0,1],\mathcal O_n)$  of the Cuntz algebra, and note that its irreducible representations factor through point evaluations on $(0,1]$ and so cannot be QD, since $\mathcal O_n$ is not QD, also the cone embeds into an AF algebra which is inner-QD (actually strong QD) \cite{Oza03}. For the lack of passing to quotients, consider $C^*(\mathbb{F}_2)$, the maximal group \C-algebra of free group on two generators, which is RFD  \cite{Cho80} and $\mathcal{O}_2$ arises as its quotient.}.

One of the central problems in approximation theory of traces on \C-algebras is the question that in what circumstances all amenable traces are quasidiagonal \cite[7.3.2]{Win16}, \cite{BCW}. The first major breakthrough was a result of Tikuisis-White-Winter, who showed that faithful traces on separable nuclear $C^*$-algebras in the UCT class are quasidiagonal \cite{TWW}. In particular, as examples in \cite{DCE} and \cite{Sch20} suggest, it is plausible that for a discrete group $G$, the full $C^*$-algebra $C^*(G)$ is inner quasidiagonal provided that $G$ is amenable. A partial answer to this question is provided by {\it Proposition D} below.

The main objective of the current paper  is to  give affirmative  answers to some natural  questions along the above lines.  First, we obtain a  new characterization of locally finite dimensional traces, as follows:

\vspace{.3cm}
{\bf Theorem A.}
	A tracial state $\tau$ on a simple \C-algebra $A$ is locally finite dimensional iff for every $\epsilon>0$, every positive integer $n\geq 1$, and every $a_1,..., a_n\in A$, there are disjoint irreducible representations $\pi_j:A\to B(H_j)$ and finite rank projections $p_j\in B(H_j)$, for $1\leq j\leq n$,  such that 
	$$\Big|\tau(a_i)-\frac{\sum_{j=1}^{n}{Tr(p_j\pi_j(a_i))}}{\sum_{j=1}^{n}{Tr(p_j)}}\Big|<\epsilon,$$ for each $1\leq i\leq n$,	and
	$$\max_{1\leq j\leq n}{\|[p_j,\pi_j(a_i)]\|}<\epsilon,$$
where $Tr$ is the canonical trace on the matrix algebra $p_jB(H_j)p_j$ i.e. the sum of eigenvalues counted with multiplicities.

\vspace{.3cm}
This is then used to prove that for simple \C-algebras, these tracial states form a convex set, partially answering an open question raised by Brown \cite[page 29]{Bro03}. 

\vspace{.3cm}
{\bf Corollary B.} The set of all locally finite dimensional traces on a simple separable \C-algebra form a convex set.

\vspace{.3cm}
To highlight the importance of LFD traces, we observe that if for each separable \C-algebra $A$, uniformly amenable QD traces are LFD,  it follows that  hyperfinite type II$_1$-factor $\mathcal{R}$ is QD, something which seems to be wide-open at present. 

\vspace{.3cm}
{\bf Proposition C.} The following statements are equivalent:

$(i)$	For every separable \C-algebra $C$, $UAT(C)\cap AT_{QD}(C)\subset AT_{LFD}(C)$,

$(ii)$  The hyperfinite type $II_1$-factor $\mathcal{R}$ is quasidiagonal, and for every unital inclusion of separable \C-algebras $B\subset A$ if $\tau\in UAT(A)\cap AT_{LFD}(A)$ then $\tau|_B\in AT_{LFD}(B)$.

\vspace{.3cm}
 Our last result gives a characterization of inner quasidiagonality of the reduced group \C-algebra of a discrete group, giving an inner-QD analog of the celebrated Rosenberg conjecture, already settled  by Tikuisis-White-Winter in \cite{TWW}. Indeed, we prove a little more.
 
\vspace{.3cm}
{\bf Proposition D.}
 Let $\Gamma$ be a discrete amenable ICC group, then all traces on the reduced \C-algebra $C^*_r(\Gamma)$ are automatically locally finite dimensional, and in particular,  $C^*_r(\Gamma)$ is strong-NF.

\section{Preliminaries and conventions}
In this section we recall a number of preliminary facts on approximation properties of traces and inner quasidiagonality, needed in the rest of this paper.
In this paper, traces are always meant to be tracial states and $T(A)$ denotes the collection of all traces. We write $\mathbb{M}_n$ for the \C-algebra of all $n\times n$ complex matrices with canonical non normalized tracial functional $Tr$. We denote the normalized trace by $tr$ or more specifically by $tr_n$. For $x\in \mathbb M_n$, we write $\|x\|_2:=tr_n(x^*x)^{\frac{1}{2}}$. Also, for a \C-algebra $A$, we denote the positive cone of $A$ by $A_+$. 

When two representations $\pi:A\to B(H)$ and $\sigma:A\to B(K)$ are (unitarily) equivalent  via a unitary $v:H\to K$ we write $\pi\sim_v\sigma$. For a representation $\pi: A\to B(H)$ and projection $p\in B(H)$, $A_p$ is the set of all $a\in A$ with $p\pi(a)=\pi(a)p$. 

For a non-degenerate representation $\rho$, we denote the central cover of $\rho$ by $c(\rho)$. This is defined to be the central projection $c(\rho):=1_{A^{**}}-e_\rho$, where $e_\rho$ is the unit of the $W^*$-algebra $ker(\rho)^{**}$. It is known that $c(\rho)A^{**}$ is isomorphic to $\rho(A)''$ and  that $c(\rho)$ determines $\rho$ up to quasi-equivalence \cite[Page 6]{BO}. We say that an irreducible representation $\rho$ is GCR if its range contains  $\mathbb{K}(H_\rho).$ Finally, we say that a projection $p\in A^{**}$ is in the socle of $A^{**}$ if the cut-down $pA^{**}p$ is finite dimensional. Note that every finite rank projection in $B(H_\rho)$ actually lives in the socle of $A^{**}$.

In the rest of this section we recall the definition of certain classes of traces and use the same notations as of the ones used by Brown in \cite{Bro03}.

\begin{definition}\label{amt}
	A trace $\tau$ on a $C^*$-algebra $A$ is {amenable} if there exists a net of c.c.p. (contractive completely positive) maps $\varphi_n:A\to\mathbb{M}_{k(n)}$ such that $$\lim_{n\to\infty}{\|\varphi_n(ab)-\varphi_n(a)\varphi_n(b)\|_2}=0,\ \   \tau(a)=\lim_{n\to\infty}{tr_{k(n)}\circ\varphi_n(a)},$$ for $a,b\in A$.  An amenable  trace $\tau$ is called QD (quasidiagonal) if  moreover  $$\lim_{n\to\infty}{\|\varphi_n(ab)-\varphi_n(a)\varphi_n(b)\|}=0,$$for $a,b\in A,$ in the operator norm. Finally, an amenable trace $\tau$ on a \C-algebra $A$ is called LFD (locally finite dimensional) if there are c.c.p. maps  $\varphi_n:A\to\mathbb{M}_{k(n)}$ with $$d(a, A_{\varphi_n})\to 0,\ \ \ \tau(a)=\lim_{n\to\infty}{tr_{k(n)}\circ\varphi_n(a)},$$ for $a\in A$, where $A_{\varphi_n}$ is the multiplicative domain of $\varphi_n$.
\end{definition}	
When $A$ is unital,  in the above definition, one could always arrange the required maps to be unital, and when $A$ is separable, there is a sequence (instead of a net) witnessing the above approximations.
	The sets of amenable, QD, and LFD traces are denoted by $AT(A)$, $AT(A)_{QD}$, and $AT(A)_{LFD}$, respectively. 
	
Next, we recall the notion of QD \C-algebras, introduced and studied first by Javier Thayer \cite{Th77} and later by Norberto Salinas \cite{Sa77}, Dan Voiculescu \cite{Voi91}, \cite{Voi93}. It includes the natural subclass of inner-QD \C-algebras of Blackadar-Kirchberg \cite{BK01}.

\begin{definition} A separable \C-algebra $A$ is QD (quasidiagonal) if there exists a sequence of c.c.p. maps $\varphi_n:A\to\mathbb{M}_{k(n)}$ such that $$\lim_{n\to\infty}{\|\varphi_n(ab)-\varphi_n(a)\varphi_n(b)\|}=0,\ \ \lim_{n\to\infty}{\|\varphi_n(a)\|}=\|a\|,$$ for $a,b\in A,$ and it is called inner-QD (inner quasidiagonal) if for each $\epsilon>0$, each $n\geq 1$, and every $x_1,...,x_n\in A$, there is a projection $p$ in the socle of $A^{**}$ with $\|px_jp\|>\|x_j\|-\epsilon$ and $\|[p,x_j]\|<\epsilon$, for each $1\leq j\leq n$.
\end{definition}

First let us recall a characterization of inner quasidiagonality, due to Blackadar and Kirchberg (cf., \cite[Corollary 11.3.7]{BO}).

\begin{proposition}\label{obinnqd}
	 A separable $C^*$-algebra $A$ is inner-QD if and only if there is a sequence of c.c.p. maps $\varphi_n:A\to\mathbb{M}_{k(n)}$ such that $\|a\|=\lim\|\varphi_n(a)\|$ and $d(a, A_{\varphi_n})\to 0$, for  $a\in A$.
\end{proposition}

For a free ultrafilter $\omega$ on the set of natural numbers, let  $\mathcal{Q}_\omega$ be the ultrapower of the universal UHF algebra $\mathcal{Q}$. Let $q_\omega:\ell^\infty(\mathcal{Q})\to\mathcal{Q}_\omega$ be the canonical quotient map and define the canonical trace $\tau_{\mathcal{Q}_\omega}$ on $\mathcal{Q}_\omega$ by $$\tau_{\mathcal{Q}_\omega}(q_\omega(a_1,a_2,...))=\lim_{n\to\omega}{\tau_{\mathcal{Q}}(a_n)}.$$
This is then the unique tracial state on $\mathcal{Q}_\omega$  (\cite[Theorem 8]{Oza13}). A c.c.p. map $\tilde{\Phi}:A\to\mathcal{Q}_\omega$ is said to be liftable if there is a c.c.p. map $\Phi:A\to\ell^\infty(\mathcal{Q})$ such that $q_\omega\circ\Phi=\tilde{\Phi}$.\\

Next we recall further characterizations of QD \C-algebras and traces in terms of liftability (see; \cite[Proposition 1.4]{TWW}, \cite[Theorem 4]{LL16}, and \cite[Proposition 3.4]{Gab}).

\begin{proposition} \label{lift}
	For a separable \C-algebra $A$,
	
	($i$) $A$ is quasidiagonal iff there is a liftable *-monomorphism $\Phi:A\to\mathcal{Q}_\omega$.\\
	
	($ii$) If $A$ is unital then $A$ is inner QD iff there is a sequence of u.c.p. maps $\varphi_n:A\to\mathbb{M}_{k(n)}$ with irreducible minimal Stinespring's dilation such that the induced map $\Phi:A\to\frac{\prod\mathbb{M}_{k(n)}}{\sum\mathbb{M}_{k(n)}}$ is a faithful $*$-homomorphism.\\
	
	$(iii)$ A trace $\tau$ on $A$ is QD iff there is a liftable *-homomorphism $\Phi:A\to\mathcal{Q}_\omega$ with $\tau=\tau_{\mathcal{Q}_\omega}\circ\Phi$.
\end{proposition}

We have the following {\it de facto} result which we omit its straightforward proof (using Propositions \ref{obinnqd} and \ref{lift}).

\begin{proposition} \label{inner}
	Let $A$ be a separable \C-algebra admitting a faithful LFD trace then $A$ is inner-QD. Moreover, any unital, inner-QD \C-algebra has at least one LFD trace.
\end{proposition}

By a well-known result of Voiculescu, the cone of every \C-algebra is quasidiagonal (more is true, as quasidiagonality is homotopy invariant). An analogous result in the context of QD traces is known as "Gabe's order-zero quasidiagonality" (cf., \cite[Proposition 3.2]{BCW}).

\begin{proposition} \label{gabe}
	Let $A$ be a separable \C-algebra and $\tau\in AT(A)$ then there is a c.c.p. order-zero map $\Phi:A\to\mathcal{Q}_\omega$ which is liftable and $\tau(a)=\tau_{\mathcal{Q}_\omega}(\Phi(a)\Phi(1)^{n-1})$ for $n=1,2,\cdots$ and $a\in A$. In particular, every amenable trace on the cone of $A$ is quasidiagoanl.
\end{proposition}

\section{Proof of the main results}\label{lfd}

In this section we show  that the set of LFD traces is convex. For this, we need a characterization of LFD traces in terms of irreducible representations (to control the multiplicative domain of the underlying c.c.p. maps). First we recall a result of Blackadar-Kirchberg \cite{BK01}, which gives a way to calculate the distance to the multiplicative domain (cf., \cite[Proposition 11.3.6]{BO}).

\begin{lemma} \label{mult}	For a separable \C-algebra $A$,
	 let $p\in A^{**}$ be in the socle and $A_p$ be the multiplicative domain of the map $a\mapsto pap$. Then $d(a,A_p)=\|[a,p]\|$, for each $a\in A$.
\end{lemma}

Next, let us prove the first main result of this paper.  

\vspace{.3cm}
{\it Proof of Theorem A.}
	 Let $\tau$ be a LFD trace, and for $\epsilon>0$, and contractions $a_1,...,a_n\in A$, choose positive integer $N\geq 1$ and a c.c.p. map $\varphi:A\to\mathbb{M}_N$ with $$d(a_i,A_\varphi)<\epsilon/4,\ \  \big|\tau(a_i)-\frac{Tr(\varphi(a_i))}{N}\big|<\epsilon,$$ for $1\leq i\leq n$. Choose contractions $b_1,...,b_n\in A_\varphi$ such that $\|a_i-b_i\|<\epsilon/2$. Since the restriction $\varphi|_{A_\varphi}$ of $\varphi$ is a finite dimensional representation, it decomposes into a direct sum of finitely many irreducible representations of finite dimension say, $\varphi|_{A_\varphi}=\sigma_1\oplus...\oplus\sigma_r$,  by \cite[Theorem 5.5.1.]{Mu90} there are irreducible representations $\tilde{\sigma}_i:A\to B(\tilde{H_i})$ and closed vector subspaces $H'_i\subset\tilde{H_i}$ which are invariant for $\sigma_i(A_\varphi)$ such that the restriction of $\tilde{\sigma}_i$ to $A_\varphi$ and Hilbert space $H'_i$ is unitarily equivalent to $\sigma_i$. After  obvious identifications, let $p_i$ be the corresponding orthogonal projection of $\tilde{H}_i$ onto $H_i$ which has a finite rank as $dim H_i$ is finite. Since $A$ is simple, each $\tilde{\sigma}_i$ is non-GCR, hence by \cite[Theorem A.2]{BK01} there are uncountably many disjoint irreducible representations say $\pi_i$ with the same kernel as $\tilde{\sigma}_i$ (hence approximate unitarily equivalent by Voiculescu's Theorem \cite[Theorem II.5.8.]{Dav96}). Let $u_i$ be a unitary operator implementing $\|\tilde{\sigma}_i(b_j)-u_i^*\pi_i(b_j)u_i\|<\epsilon/3$ for $j=1,2,...,n$. Put $q_i:=u_ip_iu_i^*$, then,
	 $$\Big|\tau(b_j)-\frac{\sum_{i=1}^{r}{Tr(q_i\pi_i(b_j))}}{\sum_{i=1}^{r}{Tr(q_i)}}\Big|<3\epsilon,$$
	 $$\max_{i}{\|[q_i, \pi_i(b_j)]\|}<\epsilon,$$
	 for $j=1,2,...,n.$
	
	Conversely, let such representations $\pi_j$ and projections $p_j$ exist and put $$N:=\sum_{j=1}^{n}{Tr(p_j)},\ \  \varphi(x):=\bigoplus_{j=1}^n{p_j\pi_j(x)p_j}, \ \ p:=p_1+...+p_n.$$ Let us identify each $B(H_j)=\pi_j(A)''$ with a unique summand of $A^{**}$, then each $p_j$ lives in the socle of $A^{**}$. Since the representations $\pi_j$ are disjoint, these summands are different, hence the projections $p_j$ are orthogonal and sum to a projection $p$ in the socle of $A^{**}$. Now $A_\varphi$ consists of those elements  commuting with projections $p_j$ and hence with $p$, therefore, $A_\varphi\subset A_p$. By Lemma \ref{mult}, $$d(a,A_p)=\|[a,p]\|=\max_{1\leq j\leq n}\|[p_j,\pi_j(a)]\|,$$ and we are done.

\vspace{.3cm}

Blackadar and Kirchberg introduced the notion of pure matricial states in \cite{BKirredinnQD}, here we need an analogous notion.
\vspace{.3cm}
\begin{definition}
	Let $A$ be a \C-algebra, $\mathfrak{F}\subset A$ be a finite set and $\epsilon>0$. We call $f\in A^*$ an $(\mathfrak{F},\epsilon)$-ATS (approximate tracial state) if there are finitely many disjoint irreducible representations $\sigma_i:A\to B(H_i)$, $i=1,2,...,n$ and  finite rank projections $p_{i}\in B(H_i)$, such that $f(a)=\frac{\sum_{i}{Tr(p_{i}\sigma_i(a))}}{\sum_{i}{Tr(p_{i})}}$, and $\max_{i}{\|[p_{i}, \sigma_i(x)]\|}<\epsilon$, for every $a\in A$ and $x\in\mathfrak{F}$.
\end{definition}
 \vspace{.3cm}
The following lemma is the key technical tool to prove convexity of $AT_{LFD}$.

\vspace{.3cm}
\begin{lemma}\label{ATMS}
$(i)$ For disjoint irreducible representations $\sigma_i:A\to B(H_i)$, $i=1,2,...,n$, and  finite rank projections $p_{i}\in B(H_i)$ as above, the map $$\varphi: A\to B(\bigoplus_{i} p_{i}H_i);  \ \ \varphi(a):=\oplus_{i=1}^{n}{{p_{i}\sigma_i(a)}p_{i}}\ \ (a\in A),$$ is  c.c.p.  with $d(x,A_\varphi)<\epsilon$, for all $x\in\mathfrak{F}.$

$(ii)$ If $A$ is simple and $f$ and $g$ are  $(\mathfrak{F}, \epsilon)$-ATS, for $\delta>0$ there is a $(\mathfrak{F},\epsilon+\delta)$-ATS, $h$, with $|h(x)-\frac{1}{2}(f+g)(x)|<\delta$, for all $x\in\mathfrak{F}$. 
\end{lemma}
\begin{proof}
The first assertion of part $(i)$ directly follows from the definition, and the second  follows from Lemma \ref{mult}. For part $(ii)$, let $f(x)=\frac{\sum_{i}{Tr(p_{i}\sigma_i(x))}}{\sum_{i}{Tr(p_{i})}}$ and $g(x)=\frac{\sum_{r}{Tr(p'_{r}\sigma'_r(x))}}{\sum_{r}{Tr(p'_{r})}}$, with $\max_{i}{\|[p_{i}, \sigma_i(x)]\|}<\epsilon, \ \ \max_{r}{\|[p'_{r}, \sigma'_r(x)]\|}<\epsilon,$ for $x\in\mathfrak F$ and $i=1,2,...,n$, $r=1,2,...,n'$. Set $d:=\sum_{i}{Tr(p_{i})}$ and $d':=\sum_{r}{Tr(p'_{r})}$. Then,
$$\frac{1}{2}(f(x)+g(x))=\frac{d'\sum_{i}{Tr(p_{i}\sigma_i(x))}+d\sum_{r}{Tr(p'_{r}\sigma'_r(x))}}{2dd'},$$
for  $x\in\mathfrak{F}.$ Since $A$ is simple, there are mutually disjoint irreducible representations $\{\pi_{i,j}\}$, $i=1,...,n$, $j=1,...,d'$, and $\{\pi'_{r,s}\}$, $r=1,...,n'$, $s=1,...,d$ such that $\pi_{i,j}$ are approximate unitarily equivalent to $\sigma_i$ and the same happens for $\pi'_{r,s}$ and $\sigma'_r$, that is, there are families  $\{u_{i,j}\}$ and $\{v_{r,s}\}$ of unitary operators such that $$\|\sigma_i(x)-u_{i,j}^*\pi_{i,j}(x)u_{i,j}\|<\delta/2, \ \ \|\sigma'_{r,s}(x)-v_{r,s}^*\pi'_{r,s}(x)v_{r,s}\|<\delta/2,\ \ (x\in\mathfrak{F}).$$ Then, for $q_{i,j}:=u_{i,j}p_iu_{i,j}^*$ and $q'_{r,s}:=v_{r,s}p'_rv_{r,s}^*$, 
$$|Tr(p_i\sigma_i(x))-Tr(q_{i,j}\pi_{i,j}(x))|<\delta Tr(q_{i,j})=\delta Tr(p_i),$$
$$|Tr(p'_r\sigma'_r(x))-Tr(q'_{r,s}\pi'_{r,s}(x))|<\delta Tr(p'_r),$$
and
$$\|[q_{i,j},\pi_{i,j}(x)]\|\leq\|[p_i,\sigma_i(x)]\|+2\|\sigma_i(x)-u_{i,j}^*\pi_{i,j}(x)u_{i,j}\|<\epsilon+\delta,$$
for  $x\in\mathfrak{F}.$ Similarly,
$$\|[q'_{r,s},\pi'_{r,s}(x)]\|<\epsilon+\delta, \ \ (x\in\mathfrak{F}).$$
To finish the proof let us put, 

\begin{align*}
h(a)&:=\frac{\sum_{i,j}{Tr(q_{i,j}\pi_{i,j}(a))}+\sum_{r,s}{Tr(q'_{r,s}\pi'_{r,s}(a))}}{\sum_{i,j}{Tr(q_{i,j})}+\sum_{r,s}{Tr(q'_{r,s})}}\\&=\frac{\sum_{i,j}{Tr(q_{i,j}\pi_{i,j}(a))}+\sum_{r,s}{Tr(q'_{r,s}\pi'_{r,s}(a))}}{2dd'},
\end{align*}
for $a\in A$. Then,
\begin{align*}
|h(x)-\frac{1}{2}(f(x)+g(x))|&\leq
\frac{1}{2dd'}\Big(\sum_{i,j}{\Big|Tr(q_{i,j}\pi_{i,j}(x))-Tr(p_i\sigma_i(x))\Big|}\\&+\sum_{r,s}{\Big|Tr(q'_{r,s}\pi'_{r,s}(x))-Tr(p'_r\sigma'_r(x))\Big|}\Big)\\&\leq\frac{1}{2dd'}(\sum_{i,j}{\delta Tr(p_i)}+\sum_{r,s}{\delta Tr(p'_r)})<\delta,
\end{align*}
for all $x\in\mathfrak{F},$ as required.
\end{proof}

The convexity  now follows.

\vspace{.3cm}
{\it Proof of Corollary B.}\label{cnvx} Given LFD traces $\tau$, $\tau'$, $\epsilon>0,$ and a finite set $\mathfrak{F}$ of contractions,  $\tau$ and $\tau'$  are  $(\mathfrak{F},\epsilon/3)$-ATS by {\it Theorem A}. Hence, $\frac{1}{2}(\tau+\tau')$ is  $(\mathfrak{F},\epsilon)$-ATS, by Lemma \ref{ATMS}$(ii)$ for $\delta=\epsilon/3$, and so is LFD by {\it Theorem A}.  Now the set of LFD traces is closed under taking dyadic convex combinations, and so is convex by weak$^*$-closedness.    

\vspace{.3cm}
Next we prove {\it Proposition C}. This follows from Gabe's order-zero quasidiagonality for separable subalgebras $A$ of $\mathcal{R}$ and equivalence of inner quasidiagonality for $A$ and its cone.
  
\vspace{.3cm}
{\it Proof of Proposition C.}
	$(i)\Rightarrow (ii)$. Given an inclusion $B\subset A$ and $\tau\in UAT(A)\cap AT_{LFD}(A)$, since $\tau|_B$ is quasidiagonal and  uniformly amenable, $\tau|_B\in AT_{LFD}(B)$. Let $C\subset\mathcal{R}$ be a separable simple subalgebra, then by Gabe's order-zero quasidiagonality, $C_0(0,1]\otimes C$ has a faithful, uniformly amenable, quasidiagonal trace. Indeed, $C$ has a faithful, uniformly amenable tracial state coming from the normalized tracial state of $\mathcal{R}$, and integrating this against the Lebesque measure on $(0,1]$ provides a faithful, uniformly amenable tracial state on cone of $C$. This trace is then LFD by $(i)$,  which implies that $C_0(0,1]\otimes C$ is inner quasidiagonal by {\it Corollary B}. Since inner quasidiagonality of $C$ and $C_0(0,1]\otimes C$ are equivalent \cite[Corollary 3.11]{BK01}, $C$ is quasidiagonal. Now since all separable subalgebras of $\mathcal R$ are quasidiagonal and every separable subalgebra embeds into a simple subalgebra (by simplicity of $\mathcal{R}$), it follows that $\mathcal R$ is quasidiagonal. 
	
	$(ii)\Rightarrow (i)$. Given a separable \C-algebra $C$ and $\tau\in UAT(C)\cap AT_{QD}(C)$, by a result of Brown \cite[Proposition 3.2.2.]{Bro03}, $\pi_\tau(C)''$ is hyperfinite. Since every separable, finite and hyperfinite von Neumann algebra embeds into $\mathcal{R}$ (see the proof of \cite[Theorem A.1]{MuRo2020}), we may assume without loss of generality that $\pi_\tau(C)''\subset\mathcal{R}$. Now let $\pi_\tau(C)\subset D\subset\mathcal{R}$ be a separable, simple, monotracial \C-algebra, then by quasidiagoanlity of $\mathcal{R}$, $D$ is strongly quasidiagonal with a unique LFD trace \cite[6.1.14]{Bro03}. In particular, $\tau\in AT_{LFD}(C)$.

\vspace{.3cm}
The next corollary illustrates that simplicity is rather technical obstruction. Indeed, the lamplighter group $\Gamma:=\mathbb{Z}_2\wr\mathbb{Z}$ is an amenable, ICC group such that $C^*(\Gamma)$ is RFD hence far from being simple but every trace on $C^*(\Gamma)$ is LFD.

\vspace{.3cm}
{\it Proof of Proposition D.}
 First note that the reduced \C-algebra $C^*(\Gamma)$ has a faithful irreducible representation $\pi$ such that $\pi(C^*(\Gamma))\cap\mathbb{K}(H_\pi)=0$. Indeed, $C^*_r(\Gamma)''=\mathcal{R}$ since $\Gamma$ is ICC, and by amenability of $\Gamma$, $C^*(\Gamma)=C^*_r(\Gamma)$, hence $C^*(\Gamma)$ admits a faithful, $II_1$-factor representation. Therefore, by a result of Glimm \cite[Theorem A.2]{BK01}, there are uncountably many non-GCR irreducible representations which are faithful (note that by separability, prime ideals and primitive ideals are the same \cite[II.6.5.15]{BlacC*vonNeu}). Invoking the main result of Tikusis-White-Winter \cite[Corollary 6.1]{TWW}, we conclude that every trace on $C^*(\Gamma)$ is QD. Combining this with a result of Nate Brown \cite[Proposition 3.3.2]{Bro03}, we obtain an increasing sequence $(p_n)$ of finite rank projections in $\mathbb{B}(H_\pi)$ such that for each $a\in C^*(\Gamma)$, $\|[p_n,\pi(a)]\|\to 0$, and  for every trace $\tau\in T(C^*(\Gamma))$, there is a subsequence $p_{n(k)}$ with $\tau(a)=\lim_{k\to\infty}{\frac{Tr(p_{n(k)}\pi(a))}{Tr(p_{n(k)})}}$. Therefore, $\tau$ satisfies the conditions of (the proof of) {\it Theorem A}, and so it is LFD. In particular, the (faithful) canonical trace of $C^*(\Gamma)$ is LFD. Hence $C^*(\Gamma)$ is inner-QD (by Proposition \ref{inner}) and nuclear, which is equivalent to being strong-NF.
 
\vspace{.3cm}		
{\it Remark.}$(i)$ The main result of \cite{DykTor2014} and similar argument as above reveals that for  separable, unital, residually finite dimensional \C-algebras $A_1$ and $A_2$ with $(dim (A_1)-1)(dim (A_2)-1)\geq2$, the full free product, $A=A_1*A_2$ satisfies $AT_{QD}(A)=AT_{LFD}(A)$ and every LFD trace has the form described in {\it Theorem A}.\\
$(ii)$ The simplicity assumption on $A$ is used in the proof of {\it Theorem A} to make sure
that $A$ has no GCR irreducible representation. In particular, this result is also valid in the
non simple case, as long as this extra condition holds. The same holds  for {\it Corollary B} and {\it Proposition C}. A practical instance is a non simple
$\mathcal{Z}$-stable \C-algebra, where $\mathcal{Z}$ is the Jiang-Su algebra (as easily seen by the Kirchberg slice
lemma \cite[Lemma 4.1.9]{RS02}).

\section{Acknowledgment}
We would like to thank Professor Wilhelm Winter for pointing out a gap in an earlier version of this paper. We also thank the anonymous referee for careful reading and suggestions to improve the paper. 
	
\end{document}